\documentclass[12pt]{article}
\usepackage{amsmath,amssymb,amsthm,amsfonts}
\usepackage{hyperref}
\usepackage{tikz}
\usepackage{enumerate}
\makeatletter
\def\imod#1{\allowbreak\mkern10mu({\operator@font mod}\,\,#1)}
\makeatother

\newtheorem{theorem}{Theorem}[section]

\newtheorem{lemma}{Lemma}[section]

\theoremstyle{definition}
\newtheorem{definition}{Definition}[section]

\newtheorem{example}{Example}[section]

\allowdisplaybreaks

\usepackage{graphicx}
\usepackage{placeins}
\usepackage{float}
\usepackage{subfigure}
\usepackage{color}
\usepackage{wrapfig}  
\usepackage{float}   
\usepackage{lastpage}  
\usepackage{tkz-graph}
\usepackage{tikz-qtree}

\usepackage{graphicx}
	\graphicspath{{./figures/}}
\title{Anti-Ramsey Theory on Complete Bipartite Graphs}

\author{Stephan Cho\thanks{Dept.\ of Mathematics, UC Berkeley, USA {\tt stephancho@berkeley.edu}.  Supported by NSF grant no. 1262930.} \and 
Jay Cummings\thanks{Dept.\ of Mathematics, UC San Diego, USA {\tt jjcummings@ucsd.edu}.} \and 
Colin Defant\thanks{Dept.\ of Mathematics, University of Florida, USA {\tt cdefant@ufl.edu}.  Supported by NSF grant no. 1262930.} \and 
Claire Sonneborn\thanks{University of Dayton, USA {\tt csonneborn1@udayton.edu,}.  Supported by NSF grant no. 1262930.}}

\begin{document}
\maketitle

\begin{abstract}
We consider quadruples of positive integers $(a,b,m,n)$ with $a\leq b$ and $m\leq n$ such that any proper edge-coloring of the complete bipartite graph $K_{m,n}$ contains a rainbow $K_{a,b}$ subgraph. We show that any such quadruple with $m\geq a$ and $n>(a^2-a+1)(b-1)$ satisfies this property and find an infinite sequence where this bound is sharp. We also define and compute some new anti-Ramsey numbers.
\end{abstract}

\section{Introduction} 
An edge-colored graph is said to be \emph{rainbow} if no two edges have the same color. Similarly, an edge-coloring of a graph is said to be a \emph{proper} coloring if no two adjacent edges have the same color. A typical anti-Ramsey problem concerns properly edge-coloring complete graphs $K_n$ in order to forbid or guarantee the existence of certain rainbow subgraphs. However, proper edge-colorings of complete bipartite graphs have received considerably less attention. It is our goal to prove a few basic results about proper edge-colorings of complete bipartite graphs in an anti-Ramsey-theoretic setting. Namely, we will investigate quadruples of positive integers $(a,b,m,n)$ with $a\leq b$ and $m\leq n$ such that every proper edge-coloring of $K_{m,n}$ contains at least one rainbow $K_{a,b}$ subgraph.  We denote this property by $K_{m,n} \rightarrow_R K_{a,b}$.
\par 
We will make use of the following canonical correspondence between properly edge-colored complete bipartite graphs and latin rectangles. Let $G$ be a properly edge-colored copy of $K_{m,n}$. Let $A$ be the set of $m$ nonadjacent vertices  in $G$, and let $B$ be the set of $n$ nonadjacent vertices  in $G$. Then we may construct an $m\times n$ latin rectangle $R$ so that each row of $R$ corresponds to a vertex in $A$ and each column in $R$ corresponds to a vertex in $B$. Furthermore, every symbol in $R$ corresponds to the color assigned to the edge connecting the vertices that correspond to the row and column in which that symbol is placed. The fact that no two adjacent edges in $G$ have the same color corresponds to the fact that no symbol appears more than once in any row or column of $R$. We will always convene to let an $m\times n$ latin rectangle be one with $m$ rows and $n$ columns. An $x\times y$ subrectangle of a latin rectangle $R$ is the intersection of $x$ rows and $y$ columns of $R$. Finally, we will use the word ``rainbow" to describe any subrectangle whose symbols are all distinct.    
\section{General bounds}
\begin{theorem} \label{Thm2.1} 
Let $a,b,m,n$ be positive integers such that $a\leq b$, $a\leq m\leq n$, and $n>(a^2-a+1)(b-1)$. Every properly edge-colored $K_{m,n}$ contains a rainbow $K_{a,b}$ subgraph. That is, $K_{m,n}\rightarrow_R K_{a,b}$.  
\end{theorem} 
\begin{proof} 
It suffices to show that any properly edge-colored $K_{a,n}$ contains a rainbow $K_{a,b}$ subgraph. Suppose we have a proper edge-coloring of $K_{a,n}$, and let $R$ be the corresponding $a\times n$ latin rectangle. Because $R$ is latin, we may choose any column of $R$ to obtain a rainbow $a\times 1$ subrectangle of $R$. Now, suppose that we have managed to find a rainbow $a\times t$ subrectangle of $R$ for some $t\in\{1,2,\ldots,b-1\}$, and call this subrectangle $T$. Each of the $at$ distinct symbols in $T$ may appear at most $a-1$ times outside of $T$ because it can appear no more than once in each row of $R$. Furthermore, because there are $n-t$ columns of $R$ outside of $T$ and $n-t>(a^2-a+1)(b-1)-t\geq(a^2-a+1)t-t=at(a-1)$, we see that there is some column of $R$ containing none of the $at$ symbols that appear in $T$. We may annex this additional column to $T$ to form a rainbow $a\times (t+1)$ subrectangle of $R$. By induction, we see that we may construct a rainbow $a\times b$ subrectangle of $R$, so the proof is complete.  
\end{proof} 
It is natural to ask how good the bound $n>(a^2-a+1)(b-1)$ used in Theorem \ref{Thm2.1} is. In other words, if $n\leq (a^2-a+1)(b-1)$, can we always find a proper edge-coloring of $K_{m,n}$ that forbids the appearance of any rainbow $K_{a,b}$ subgraphs? Theorem \ref{Thm2.2} makes progress toward answering this question. First, we need a couple of preliminary results.  

\begin{definition}
Let $\mathcal{P}$ be the set of all possible orders of a finite projective plane (we convene to let $1\in\mathcal P$).
\end{definition}

We will make use of a theorem of Singler \cite{Proj} from 1938, which we state now. 

\begin{theorem}\label{Sing} (Singler)
Let $q$ be a power of a prime\footnote{In this paper, we consider $1$ to be a prime power.}. The automorphism group of $PG(2,q)$ contains a cyclic subgroup $\langle \sigma \rangle$ which acts regularly on points and regularly on lines.
\end{theorem}

Singler's result allows us to prove the following important lemma. 

\begin{lemma}\label{prop2}
Let $a$ be an integer. If there exists an $a \times (a^2 - a + 1)$ latin rectangle without any rainbow $a \times 2$ subrectangles, then $a-1\in \mathcal P$. If $a-1$ is a prime power, then there exists an $a \times (a^2 - a + 1)$ latin rectangle without any rainbow $a \times 2$ subrectangles.
\end{lemma}

\begin{proof}
Suppose $R$ is an $a \times (a^2 - a + 1)$ latin rectangle without a rainbow $a \times 2$ subrectangle. Note that any two columns of $R$ have a common element.  Given any column $C$ of $R$, each of the other $a(a-1)$ columns must contain one of the $a$ elements in $C$.  Thus each of the elements of $C$ is, on average, in at least $a-1$ other columns.

Now, assume that one of these elements is in strictly fewer than $a-1$ other columns.  Then another element would have to be in at least $a$ other columns.  Clearly this is impossible -- along with $C$ such an element is in at least $a+1$ columns, implying that two must be in a common row.
\par
Thus every element is in exactly $a$ different columns and every column has exactly $a$ different elements.  Furthermore it is easy to see by the same reasoning that no two columns can intersect in more than one element.  Thus the columns of $R$ satisfy the axioms of a projective plane of order $a-1$.

We now assume that $a-1$ is a prime power and consider a projective plane $P$ of order $a-1$. We aim to construct a latin rectangle of the asserted size.

Take any line $B$ from $P$ and use its elements to form the first edge of our latin rectangle; the orders of the elements are not important.  By Theorem \ref{Sing} there exists a cyclic subgroup $\langle \sigma \rangle$ which acts regularly on the points and regularly on lines.  For each $i$, apply $\sigma$ to the $i^{\text{th}}$ member of the first column to obtain the $i^{\text{th}}$ member of the second column.  In general, obtain each successive column by applying $\sigma$ to the previous.

Recall that a projective plane of order $a-1$ has $(a-1)^2 + (a-1) + 1 = a^2 - a + 1$ points, so $\sigma$ is simply a cyclic ordering of these.  Thus, the $i^{\text{th}}$ row is simply a list of all $a^2 - a + 1$ elements which follows this ordering and starts with the $i^{\text{th}}$ member of the first column. Given this, it's clear that each row has distinct entries.  Furthermore, since $B$ has distinct members it is also clear that each successive column has distinct entries.  Indeed, if two entries were the same then that implies that the two rows are completely the same, as each successive entry is obtained by repeatedly applying $\sigma$.  In particular their two entries in $B$ are the same.  But this is a contradiction as $B$ was a single line in $P$ and hence has distinct members.  This completes the proof.
\end{proof}

\begin{example}
For the seven-point plane, consider $\sigma = (1,3,5,7,2,4,6)$ and $B = \{1,2,4\}$.  Then a corresponding latin rectangle is

\begin{center}
\begin{tabular}{ c c c c c c c }
  1 & 3 & 5 & 7 & 2 & 4 & 6\\
  2 & 4 & 6 & 1 & 3 & 5 & 7 \\
  4 & 6 & 1 & 3 & 5 & 7 & 2\\
\end{tabular}
\end{center}
\end{example}

\begin{theorem} \label{Thm2.2}
Let $a,b,m,n$ be positive integers satisfying $a\leq b$, $m\leq n$, $m<b$, and $n\leq(a^2-a+1)(b-1)$. If $a-1$ is a prime power, then it is possible to properly edge-color $K_{m,n}$ to forbid the existence of any rainbow $K_{a,b}$ subgraph. That is, if $a-1$ is a prime power, then $K_{m,n}\not\rightarrow_R K_{a,b}$. 
\end{theorem}

\begin{proof} 
Suppose $a-1$ is a prime power. By Lemma \ref{prop2}, there exists an $a\times (a^2-a+1)$ latin rectangle that contains no rainbow $a\times 2$ subrectangle. We prove that there exists an $m\times n$ latin rectangle containing no rainbow $a\times b$ or $b\times a$ subrectangle. Because $m<b$, there are no $b\times a$ subrectangles of any $m\times n$ latin rectangle, so it suffices to construct an $m\times n$ latin rectangle with no rainbow $a\times b$ subrectangles. Furthermore, it is easy to see that it suffices to construct such a rectangle for the case in which $m=b-1$ and $n=(a^2-a+1)(b-1)$. 
\par 
Let $m=b-1$ and $n=(a^2-a+1)(b-1)$. Let $L$ be an $m\times n$ latin rectangle. We first partition $L$ into $m$ subrectangles $A_1,A_2,\ldots,A_m$, each of size $m\times (a^2-a+1)$. By the Pigeonhole Principle, any $a\times b$ subrectangle of $L$ must contain at least two columns from $A_k$ for some $k\in\{1,2,\ldots,m\}$. In other words, any $a\times b$ subrectangle of $L$ contains an $m\times 2$ subrectangle of $A_k$ for some $k\in\{1,2,\ldots,m\}$. We will fill $L$ with symbols in such a manner so as to ensure that, for any $k\in\{1,2,\ldots,m\}$, there is no rainbow $m\times 2$ subrectangle of $A_k$, which will then imply the desired result. We may fill $L$ with symbols so that, for any distinct $j,k\in\{1,2,\ldots,m\}$, no symbol appears in both $A_j$ and $A_k$. This will ensure that the choice of symbols in $A_j$ does not affect where we may choose to place symbols in $A_k$ and vice versa. Therefore, it suffices to show that we may fill an $m\times (a^2-a+1)$ latin rectangle with symbols so that any two columns have a symbol in common. To do so, we simply extend the $a\times (a^2-a+1)$ latin rectangle that we assumed exists to an $m\times (a^2-a+1)$ latin rectangle.
\end{proof} 

We close with a final result on how to generate a sequence of latin rectangles to avoid larger and larger rainbow subrectangles. 

\begin{theorem}
Let $m,n,a,b,r,s$ be positive integers. Let $J$ be the $m \times n$ all-1s matrix, and let $A$ be an $m \times n$ latin rectangle consisting of positive integers whose maximum entry is $t$. Let $B$ be an $r\times s$ latin rectangle whose entries are nonnegative integers. Let
\[A'= J\otimes A+tB\otimes J=\left[ \begin{array}{cccc}
A+tB_{11}J & A+tB_{12}J & \cdots & A+tB_{1r}J \vspace{0.2cm}\\
A+tB_{21}J & A+tB_{22}J & \cdots & A+tB_{2r}J \\
\vdots & \vdots & \ddots & \vdots \\
A+tB_{r1}J & A+tB_{r2}J & \cdots & A+tB_{rr}J
\end{array} \right].\] 
Then $A'$ is an $rm\times sn$ latin rectangle. If $A$ has no rainbow $a\times b$ subrectangles, then $A'$ has no rainbow $r(a-1)+1\times s(b-1)+1$ subrectangles. In particular, if $K_{m,n} \not\rightarrow_R K_{a,b}$ then $K_{rm,rn} \not\rightarrow_R K_{r(a-1)+1,r(b-1)+1}$.
\end{theorem}
\begin{proof}
We defined $A'$ to be a block matrix with $rs$ blocks, each of the form $A+tB_{ij}J$. Because $t$ is the largest entry in $A$ and the entries of $A$ and $B$ are nonnegative integers, two blocks $A+tB_{ij}J$ and $A+tB_{k\ell}J$ cannot have any common entries unless $B_{ij}=B_{k\ell}$. Therefore, since $A$ and $B$ are latin, it is easy to see that $A'$ must also be latin. Let $T$ be an $r(a-1)+1\times s(b-1)+1$ subrectangle of $A'$. By the pigeonhole principle, there must be $a$ rows of $T$ and $b$ columns of $T$ which intersect in a single block of $A'$, say $A+tB_{uv}J$. The intersection of these rows and columns forms an $a\times b$ subrectangle $R$ of $A+tB_{uv}J$. Since $A$ contains no rainbow $a\times b$ subrectangles, $A+tB_{uv}J$ cannot contain a rainbow $a\times b$ subrectangle. This implies that $R$ is not rainbow. Since $R$ is a subrectangle of $T$, $T$ cannot be rainbow. As $T$ was arbitrary, this shows that $A'$ contains no rainbow $r(a-1)+1\times s(b-1)+1$ subrectangles.
\par 
Now, suppose $K_{m,n}\not\rightarrow_R K_{a,b}$. Then we may let $A$ be an $m\times n$ latin rectangle that contains no rainbow $a\times b$ subrectangles \emph{and} contains no rainbow $b\times a$ subrectangles. Setting $r=s$ in the first part of the theorem, we obtain an $rm\times rn$ latin rectangle $A'$ that contains no rainbow $r(a-1)+1\times r(b-1)+1$ subrectangles. However, since $A$ also has no rainbow $b\times a$ subrectangles, we may interchange the roles of $a$ and $b$ in the first part of the theorem to see that $A'$ also has no rainbow $r(b-1)+1\times r(a-1)+1$ subrectangles. Hence, $K_{rm,rn}\not\rightarrow_R K_{r(a-1)+1,r(b-1)+1}$.     
\end{proof}

\section{Anti-Ramsey Numbers}
Returning to Theorem \ref{Thm2.2},
we find it particularly interesting to consider the case $a=2$, $n=2m=2b$. That is, we wish to find positive integers $m$ such that any proper edge-coloring of $K_{m,2m}$ contains a rainbow $K_{2,m}$ subgraph. Theorem \ref{Thm2.1} shows that any proper edge-coloring of $K_{2,4}$ contains a rainbow $K_{2,2}$ subgraph, and the following theorem  deals with the case $m=3$. 
\begin{theorem} \label{Thm2.3}
Any properly edge-colored $K_{3,6}$ contains a rainbow $K_{2,3}$ subgraph.  
\end{theorem} 
\begin{proof} 
We prove the equivalent statement that every $3\times 6$ latin rectangle contains a rainbow $2\times 3$ or $3\times 2$ subrectangle. To do so, suppose there exists some latin rectangle $R$ with no $2\times 3$ or $3\times 2$ subrectangle. We will refer to the symbols in $R$ as "colors" in order to maintain the correspondence between $R$ and a properly edge-colored $K_{3,6}$. We will let $[r_1,r_2:c_1,c_2,c_3]$ denote the $2\times 3$ subrectangle of $R$ that is the intersection of rows $r_1$ and $r_2$ and columns $c_1$, $c_2$, and $c_3$. Let us denote the color in the $i^{th}$ row and the $j^{th}$ column of $R$ by $R_{ij}$. Note that we may swap any two columns of $R$ without changing the fact that $R$ does not contain a rainbow $2\times 3$ or $3\times 2$ subrectangle. We will let $S(i,j)$ denote the operation of swapping columns $i$ and $j$ of $R$. Furthermore, at any time, we may exchange any two colors $r$ and $s$ so that all entries of $R$ colored $r$ are recolored $s$ and vice versa. Let $C(r,s)$ denote the operation of exchanging colors $r$ and $s$, and note that this operation does not change the fact that there is no rainbow $2\times 3$ or $3\times 2$ subrectangle of $R$. Even after swapping columns and exchanging colors of $R$, we will continue to refer to the rectangle as $R$. 
\par 
Now, call any coloring of $R$ with the property that $R_{1j}=j$ for all $j\in\{1,2,\ldots,6\}$ a ``primal" coloring. Without loss of generality, we may assume $R$ is primally colored. Consider the $2\times 3$ subrectangle [1,2:1,2,3] of $R$. Because this subrectangle is not rainbow, one or more of the following equalities must hold: 
\[R_{22}=1, R_{23}=1, R_{21}=2, R_{23}=2, R_{21}=3, R_{22}=3.\] 
Suppose $R_{23}=1$. Then, performing the operation $S(2,3)$ followed by $C(2,3)$, we reach a primal coloring in which $R_{22}=1$. Next, suppose that $R_{21}=2$. Performing the operation $S(1,2)$ followed by $C(1,2)$, we reach a primal coloring in which $R_{22}=1$. A similar argument shows that we may assume, without loss of generality, that $R$ is primally colored and $R_{22}=1$. Now, consider the $2\times 3$ subrectangle $[1,2:3,4,5]$. By the same argument as before, we see that, without loss of generality, we may assume that $R$ is primally colored, $R_{22}=1$, and $R_{24}=3$. This is because, in order to ensure that $R$ is primally colored with $R_{24}=3$, we only need to use some combination of some of the operations $S(3,4)$, $S(3,5)$, $S(4,5)$, $C(3,4)$, $C(3,5)$, and $C(4,5)$, none of which change the fact that $R_{22}=1$. Consider the subrectangle $[1,2:1,3,6]$ of $R$. Because this subrectangle is not rainbow, we require either $R_{21}=6$ or $R_{23}=6$. If $R_{23}=6$, perform the operations $S(1,3)$, $S(2,4)$, $C(1,3)$, and $C(2,4)$ to obtain a primal coloring of $R$ in which $R_{22}=1$, $R_{24}=3$, and $R_{21}=6$. If $R_{21}=6$, then we do not need to perform any operations to obtain such a coloring. If we now consider the subrectangle $[1,2:1,3,5]$, it is easy to see that we must have $R_{23}=5$. Considering $[1,2:2,3,6]$, we see that we must have $R_{25}=4$. 
\par 
If we now consider $[1,3:3,4,5]$, we see that we must have $R_{33}=4$, $R_{34}=5$, or $R_{35}=6$. No matter what, there must be some column $A$ of $R$ that contains the colors $3$, $4$, and $5$. Similarly, if we consider $[1,3:1,2,6]$, we see that we must have $R_{31}=2$, $R_{32}=6$, or $R_{36}=1$. No matter what, there must be some column $B$ of $R$ containing the colors $1$, $2$, and $6$. However, this is a contradiction because the union of $A$ and $B$ is a rainbow $3\times 2$ subrectangle of $R$.
\end{proof} 

Given $K_{a,b}$ we wish to find the ``smallest" complete bipartite graph $K_{m,n}$ for which $K_{m,n} \rightarrow_R K_{a,b}$.  Here we use the number of vertices $m+n$ as our ``smallest" metric.  To this end we define the following.

\begin{definition}
Let $K_{m,n}$ be the complete bipartite graph with the fewest vertices for which $K_{m,n} \rightarrow_R K_{a,b}$.  Then we define the  \emph{vertex anti-Ramsey number} $AR_{V} \left(K_{a,b}\right)$ to be the number of vertices $m+n$ of this minimal example.
\end{definition}

We begin the study of these numbers by finding some specific values.

\begin{theorem} \label{ThmK23} 
For any integer $b\geq 2$, \[AR_V(K_{2,b})=3b.\]
\end{theorem}

\begin{proof}
By Theorem \ref{Thm2.1}, $K_{2,3b-2}\rightarrow_R K_{2,b}$, so $AR_V(K_{2,b})\leq 3b$. Now, let $m$ and $n$ be positive integers with $m+n<3b$. We wish to show that $K_{m,n}\not\rightarrow_R K_{2,b}$. We may assume $2\leq m\leq n$. If $m<b$, then we may use Theorem \ref{Thm2.2} (with $a=2$) to deduce that $K_{m,n}\not\rightarrow_R K_{2,b}$. Therefore, let us assume $m\geq b$. Since $m+n<3b$, $n<2b$. Since $m\leq n$, it is possible to properly edge-color $K_{m,n}$ with $n$ colors. Such a coloring must necessarily forbid the existence of a rainbow $K_{2,b}$ subgraph because any $K_{2,b}$ subgraph has $2b$ edges. 
\end{proof}

\begin{definition}
Let $K_{m,n}$ be a complete bipartite graph with the fewest edges for which $K_{m,n} \rightarrow_R K_{a,b}$.  Then we define the  \emph{edge anti-Ramsey number} $AR_{E} \left(K_{a,b}\right)$ to be the number of edges $mn$ of this minimal example.
\end{definition}

Note that this generalizes the notion of size anti-Ramsey numbers given in \cite{size}.  From Theorem \ref{ThmK23} we know that $AR_{E} \left(K_{2,3}\right) \leq 18$.  We initiate the study of these size anti-Ramsey numbers with the following theorem. 

\begin{theorem} 
If $a$ and $b$ are integers such that $a-1$ is a prime power and $b\geq a(a-1)$, then \[AR_E(K_{a,b})=a^2(a-1)(b-1)+ab.\]
\end{theorem} 
\begin{proof}
Suppose $a$ and $b$ are integers with $a-1$ a prime power and $b\geq a(a-1)$. By Theorem \ref{Thm2.1}, $K_{a,(a^2-a+1)(b-1)+1}\rightarrow_R K_{a,b}$, so \[AR_E(K_{2,b})\leq a((a^2-a+1)(b-1)+1)=a^2(a-1)(b-1)+ab.\] Now, let $m$ and $n$ be positive integers with $mn<a^2(a-1)(b-1)+ab$. We wish to show that $K_{m,n}\not\rightarrow_R K_{a,b}$. We may assume $a\leq m\leq n$. Since $m\leq n$, it is possible to properly edge-color $K_{m,n}$ with $n$ colors. If $n<ab$, then such a coloring must necessarily forbid the existence of a rainbow $K_{a,b}$ subgraph because any $K_{a,b}$ subgraph has $ab$ edges. Therefore, let us assume that $n\geq ab$. Since $mn<a^2(a-1)(b-1)+ab$ and $a(a-1)\leq b$, \[m<\frac{a^2(a-1)(b-1)+ab}{n}\leq\frac{ab(b-1)+ab}{n}\leq b.\] Hence, we may use Theorem \ref{Thm2.2} to conclude that $K_{m,n}\not\rightarrow_R K_{a,b}$.  
\end{proof}  
\section{Acknowledgments}
The authors would like to thank Professor Peter Johnson for posing a problem that burgeoned into this paper. 
\bigskip
\hrule

\noindent 2010 {\it Mathematics Subject Classification}:  Primary 05C15; Secondary 05B15.

\noindent \emph{Keywords: } Graph coloring; bipartite; rainbow; edge-coloring; latin square.


\begin{thebibliography}{99}
\bibitem{size}
Axenovich, M., Knauer, K., Stummp, J., Ueckerdt, T., Online and size Anti-Ramsey numbers, J. Comb. 5 (2014), no. 1, 87--114.
\bibitem{Proj} J. Singer, A theorem in finite projective geometry and some applications
to number theory, Trans. Amer. Math. Soc. 43 (1938), 377Ð385.
\bibitem{Hardy}Hardy, G. H. Ramanujan: Twelve Lectures on Subjects Suggested by His Life and Work, 3rd ed. New York: Chelsea, 1999.
\end{thebibliography}
\end{document}